\documentclass[reqno]{amsart}

\usepackage[utf8]{inputenc}
\usepackage{geometry}
\usepackage{graphicx}
\usepackage{subcaption}
\usepackage{amsmath}
\usepackage{amsthm}
\usepackage{amssymb}
\usepackage{hyperref}
\usepackage{xcolor}
\usepackage{biblatex}
\usepackage{mathtools}
\usepackage[shortlabels]{enumitem}
\usepackage{mathrsfs}
\usepackage{pgfplots}

\graphicspath{ {images/} }

\newcommand{\dtv}{d_{\mathrm{TV}}}
\newcommand{\Ex}{\mathbb{E}}
\newcommand{\Var}{\mathrm{Var}}
\renewcommand{\Pr}{\mathbb{P}}

\theoremstyle{definition}
\newtheorem{theorem}{Theorem}[section]

\newtheorem{lemma}[theorem]{Lemma}

\newtheorem{remark}{Remark}

\numberwithin{equation}{section}

\begin{document}
\title{Cutoff in total variation for the shelf shuffle}

\author[R.~Chen]{Ray Chen}
\address[R.~Chen]{University of Washington, Department of Mathematics}
\email{raychen8@uw.edu}
\author[A.~Ottolini]{Andrea Ottolini}
\address[A.~Ottolini]{Yale University, Cowles Foundation}
\email{andrea.ottolini@yale.edu}

\begin{abstract}
  We analyze the mixing time of a popular shuffling machine known as the shelf shuffler. It is a modified version of a $2m$-handed riffle shuffle ($m=10$ in casinos) in which a deck of $n$ cards is split multinomially into
  $2m$ piles, the even-numbered piles are reversed,
  and then cards are dropped from piles proportionally to their
  sizes. We prove that $\frac{5}{4} \log_{2m} n$ shuffles are necessary and sufficient to mix in total variation, and a cutoff occurs with constant window size. We also determine the cutoff profile in terms of the total variation distance between two shifted normal random variables.

\end{abstract}

\maketitle
\section{Introduction} 
Several card shuffling machines have been designed for and used by
casinos. One such machine, dubbed a ``shelf shuffler'',
works as follows. The machine accepts a deck and sequentially deals cards
from the bottom onto one of ten shelves, each with equal probability.
Each card is placed above or below all the cards presently on its shelf,
both with equal probability. The shuffled deck is then formed by joining
the piles on the shelves in random order. The obvious question that arises is how effective this machine is at
mixing a deck of cards. Among other goals, the machine seeks to shuffle
cards more efficiently than a dealer, who of course performs repeated
riffle shuffles. Bayer and Diaconis famously proved the ``seven shuffles
theorem'' in \cite{bayer92} for the riffle shuffle, described by the Gilbert-Shannon-Reeds (GSR) model. \\ \\
The shelf shuffler was studied in depth by Diaconis, Fulman, and Holmes 
in \cite{Diaconis_2013}. They analyze an extension of the shelf shuffler
consisting of $m$ shelves. Among their many results, they prove the following composition rule: if an $m_1$-shelf shuffle followed by an $m_2$-shelf shuffle, the result is that of
a $2m_1m_2$-shelf shuffle. In particular, analyzing the long time behavior of a shelf shuffler is equivalent to understanding a shelf shuffles with arbitrary large shelves. In this paper, we analyze the classical mixing time problem: informally, given a tolerance $\varepsilon\in (0,1)$, how many times should we shelf-shuffle a deck of cards to be within $\varepsilon$ from the uniform distribution over all permutations? Our mean result complements the finding \cite{Diaconis_2013} by proving the analogue of the seven-shuffle theorem for shelf shuffle machines in total variation distance. A comparison between our main Theorem \ref{thm:main} and the exact computation for the total variation distance to stationarity in a standard deck of $52$ cards can be seen in Figure~\ref{fig:tv-exact-approx}.

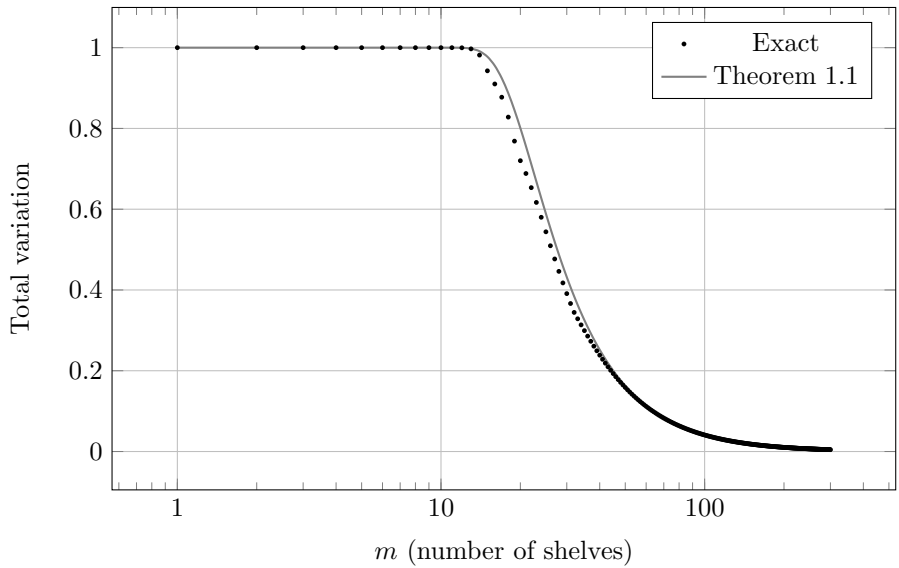
\begin{figure}
    \centering
\begin{tikzpicture}
    \begin{semilogxaxis}[
        xlabel={$m$ (number of shelves)},
        ylabel={Total variation},
        grid=major,     
        width=12cm,
        height=8cm,
        legend pos=north east,
        every axis plot/.append style={thick},
        log ticks with fixed point,
        ]

        \addplot[
            only marks,
            mark=*,
            mark size=0.5pt,
            color=black
        ] coordinates {
            (1, 1.0) (2, 1.0) (3, 1.0) (4, 1.0) (5, 1.0) (6, 1.0) (7, 1.0) (8, 1.0) (9, 1.0) (10, 1.0) (11, 0.99998) (12, 0.99969) (13, 0.99705) (14, 0.98161) (15, 0.94267) (16, 0.9101) (17, 0.87735) (18, 0.82808) (19, 0.76857) (20, 0.72009) (21, 0.6884) (22, 0.65337) (23, 0.61678) (24, 0.58001) (25, 0.54403) (26, 0.50949) (27, 0.47678) (28, 0.44609) (29, 0.4175) (30, 0.39098) (31, 0.36646) (32, 0.34443) (33, 0.32864) (34, 0.31359) (35, 0.2993) (36, 0.28575) (37, 0.27293) (38, 0.26081) (39, 0.24937) (40, 0.23856) (41, 0.22836) (42, 0.21873) (43, 0.20964) (44, 0.20105) (45, 0.19294) (46, 0.18528) (47, 0.17804) (48, 0.17118) (49, 0.1647) (50, 0.15855) (51, 0.15273) (52, 0.14721) (53, 0.14197) (54, 0.137) (55, 0.13227) (56, 0.12778) (57, 0.1235) (58, 0.11943) (59, 0.11555) (60, 0.11185) (61, 0.10833) (62, 0.10496) (63, 0.10175) (64, 0.09867) (65, 0.09574) (66, 0.09293) (67, 0.09023) (68, 0.08766) (69, 0.08519) (70, 0.08282) (71, 0.08055) (72, 0.07836) (73, 0.07627) (74, 0.07426) (75, 0.07232) (76, 0.07046) (77, 0.06867) (78, 0.06694) (79, 0.06528) (80, 0.06368) (81, 0.06214) (82, 0.06065) (83, 0.05921) (84, 0.05783) (85, 0.05649) (86, 0.0552) (87, 0.05395) (88, 0.05274) (89, 0.05158) (90, 0.05045) (91, 0.04935) (92, 0.0483) (93, 0.04727) (94, 0.04628) (95, 0.04532) (96, 0.04439) (97, 0.04348) (98, 0.04261) (99, 0.04176) (100, 0.04093) (101, 0.04013) (102, 0.03935) (103, 0.0386) (104, 0.03786) (105, 0.03715) (106, 0.03646) (107, 0.03578) (108, 0.03513) (109, 0.03449) (110, 0.03387) (111, 0.03326) (112, 0.03267) (113, 0.0321) (114, 0.03154) (115, 0.031) (116, 0.03047) (117, 0.02995) (118, 0.02945) (119, 0.02896) (120, 0.02848) (121, 0.02801) (122, 0.02756) (123, 0.02711) (124, 0.02668) (125, 0.02626) (126, 0.02584) (127, 0.02544) (128, 0.02505) (129, 0.02466) (130, 0.02428) (131, 0.02391) (132, 0.02356) (133, 0.0232) (134, 0.02286) (135, 0.02252) (136, 0.02219) (137, 0.02187) (138, 0.02156) (139, 0.02125) (140, 0.02095) (141, 0.02065) (142, 0.02036) (143, 0.02008) (144, 0.0198) (145, 0.01953) (146, 0.01927) (147, 0.01901) (148, 0.01875) (149, 0.0185) (150, 0.01825) (151, 0.01801) (152, 0.01778) (153, 0.01755) (154, 0.01732) (155, 0.0171) (156, 0.01688) (157, 0.01667) (158, 0.01646) (159, 0.01625) (160, 0.01605) (161, 0.01585) (162, 0.01566) (163, 0.01546) (164, 0.01528) (165, 0.01509) (166, 0.01491) (167, 0.01473) (168, 0.01456) (169, 0.01439) (170, 0.01422) (171, 0.01405) (172, 0.01389) (173, 0.01373) (174, 0.01357) (175, 0.01342) (176, 0.01327) (177, 0.01312) (178, 0.01297) (179, 0.01283) (180, 0.01269) (181, 0.01255) (182, 0.01241) (183, 0.01227) (184, 0.01214) (185, 0.01201) (186, 0.01188) (187, 0.01176) (188, 0.01163) (189, 0.01151) (190, 0.01139) (191, 0.01127) (192, 0.01115) (193, 0.01104) (194, 0.01092) (195, 0.01081) (196, 0.0107) (197, 0.01059) (198, 0.01049) (199, 0.01038) (200, 0.01028) (201, 0.01018) (202, 0.01008) (203, 0.00998) (204, 0.00988) (205, 0.00978) (206, 0.00969) (207, 0.0096) (208, 0.0095) (209, 0.00941) (210, 0.00932) (211, 0.00924) (212, 0.00915) (213, 0.00906) (214, 0.00898) (215, 0.0089) (216, 0.00881) (217, 0.00873) (218, 0.00865) (219, 0.00857) (220, 0.0085) (221, 0.00842) (222, 0.00834) (223, 0.00827) (224, 0.0082) (225, 0.00812) (226, 0.00805) (227, 0.00798) (228, 0.00791) (229, 0.00784) (230, 0.00777) (231, 0.00771) (232, 0.00764) (233, 0.00758) (234, 0.00751) (235, 0.00745) (236, 0.00738) (237, 0.00732) (238, 0.00726) (239, 0.0072) (240, 0.00714) (241, 0.00708) (242, 0.00702) (243, 0.00697) (244, 0.00691) (245, 0.00685) (246, 0.0068) (247, 0.00674) (248, 0.00669) (249, 0.00663) (250, 0.00658) (251, 0.00653) (252, 0.00648) (253, 0.00643) (254, 0.00638) (255, 0.00633) (256, 0.00628) (257, 0.00623) (258, 0.00618) (259, 0.00613) (260, 0.00608) (261, 0.00604) (262, 0.00599) (263, 0.00595) (264, 0.0059) (265, 0.00586) (266, 0.00581) (267, 0.00577) (268, 0.00573) (269, 0.00568) (270, 0.00564) (271, 0.0056) (272, 0.00556) (273, 0.00552) (274, 0.00548) (275, 0.00544) (276, 0.0054) (277, 0.00536) (278, 0.00532) (279, 0.00528) (280, 0.00525) (281, 0.00521) (282, 0.00517) (283, 0.00514) (284, 0.0051) (285, 0.00506) (286, 0.00503) (287, 0.00499) (288, 0.00496) (289, 0.00493) (290, 0.00489) (291, 0.00486) (292, 0.00482) (293, 0.00479) (294, 0.00476) (295, 0.00473) (296, 0.0047) (297, 0.00466) (298, 0.00463) (299, 0.0046) (300, 0.00457)
        };
        \addlegendentry{Exact}

        \addplot[
            smooth,
            color=gray
        ] coordinates {
            (1, 1.0) (2, 1.0) (3, 1.0) (4, 1.0) (5, 1.0) (6, 1.0) (7, 1.0) (8, 1.0) (9, 1.0) (10, 1.0) (11, 0.99998) (12, 0.99964) (13, 0.99764) (14, 0.99125) (15, 0.97761) (16, 0.95527) (17, 0.92459) (18, 0.88724) (19, 0.84537) (20, 0.80107) (21, 0.75605) (22, 0.71161) (23, 0.66862) (24, 0.62765) (25, 0.589) (26, 0.55282) (27, 0.5191) (28, 0.48779) (29, 0.45879) (30, 0.43195) (31, 0.40714) (32, 0.38419) (33, 0.36296) (34, 0.34332) (35, 0.32512) (36, 0.30825) (37, 0.29259) (38, 0.27804) (39, 0.26451) (40, 0.2519) (41, 0.24015) (42, 0.22917) (43, 0.21891) (44, 0.20931) (45, 0.20031) (46, 0.19187) (47, 0.18394) (48, 0.17648) (49, 0.16946) (50, 0.16285) (51, 0.15661) (52, 0.15071) (53, 0.14514) (54, 0.13987) (55, 0.13488) (56, 0.13015) (57, 0.12566) (58, 0.1214) (59, 0.11735) (60, 0.1135) (61, 0.10983) (62, 0.10634) (63, 0.10301) (64, 0.09983) (65, 0.0968) (66, 0.0939) (67, 0.09113) (68, 0.08848) (69, 0.08595) (70, 0.08352) (71, 0.08119) (72, 0.07896) (73, 0.07682) (74, 0.07476) (75, 0.07278) (76, 0.07089) (77, 0.06906) (78, 0.06731) (79, 0.06562) (80, 0.06399) (81, 0.06242) (82, 0.06091) (83, 0.05946) (84, 0.05805) (85, 0.0567) (86, 0.05539) (87, 0.05412) (88, 0.0529) (89, 0.05172) (90, 0.05058) (91, 0.04948) (92, 0.04841) (93, 0.04737) (94, 0.04637) (95, 0.0454) (96, 0.04446) (97, 0.04355) (98, 0.04267) (99, 0.04181) (100, 0.04098) (101, 0.04017) (102, 0.03939) (103, 0.03863) (104, 0.03789) (105, 0.03717) (106, 0.03648) (107, 0.0358) (108, 0.03514) (109, 0.0345) (110, 0.03387) (111, 0.03327) (112, 0.03267) (113, 0.0321) (114, 0.03154) (115, 0.03099) (116, 0.03046) (117, 0.02994) (118, 0.02944) (119, 0.02895) (120, 0.02847) (121, 0.028) (122, 0.02754) (123, 0.02709) (124, 0.02666) (125, 0.02623) (126, 0.02582) (127, 0.02541) (128, 0.02502) (129, 0.02463) (130, 0.02426) (131, 0.02389) (132, 0.02353) (133, 0.02317) (134, 0.02283) (135, 0.02249) (136, 0.02216) (137, 0.02184) (138, 0.02153) (139, 0.02122) (140, 0.02092) (141, 0.02062) (142, 0.02033) (143, 0.02005) (144, 0.01977) (145, 0.0195) (146, 0.01923) (147, 0.01897) (148, 0.01872) (149, 0.01847) (150, 0.01822) (151, 0.01798) (152, 0.01774) (153, 0.01751) (154, 0.01729) (155, 0.01706) (156, 0.01685) (157, 0.01663) (158, 0.01642) (159, 0.01622) (160, 0.01601) (161, 0.01582) (162, 0.01562) (163, 0.01543) (164, 0.01524) (165, 0.01506) (166, 0.01488) (167, 0.0147) (168, 0.01453) (169, 0.01435) (170, 0.01419) (171, 0.01402) (172, 0.01386) (173, 0.0137) (174, 0.01354) (175, 0.01339) (176, 0.01323) (177, 0.01309) (178, 0.01294) (179, 0.0128) (180, 0.01265) (181, 0.01251) (182, 0.01238) (183, 0.01224) (184, 0.01211) (185, 0.01198) (186, 0.01185) (187, 0.01172) (188, 0.0116) (189, 0.01148) (190, 0.01136) (191, 0.01124) (192, 0.01112) (193, 0.01101) (194, 0.01089) (195, 0.01078) (196, 0.01067) (197, 0.01056) (198, 0.01046) (199, 0.01035) (200, 0.01025) (201, 0.01015) (202, 0.01005) (203, 0.00995) (204, 0.00985) (205, 0.00976) (206, 0.00966) (207, 0.00957) (208, 0.00948) (209, 0.00939) (210, 0.0093) (211, 0.00921) (212, 0.00912) (213, 0.00904) (214, 0.00895) (215, 0.00887) (216, 0.00879) (217, 0.00871) (218, 0.00863) (219, 0.00855) (220, 0.00847) (221, 0.00839) (222, 0.00832) (223, 0.00824) (224, 0.00817) (225, 0.0081) (226, 0.00803) (227, 0.00796) (228, 0.00789) (229, 0.00782) (230, 0.00775) (231, 0.00768) (232, 0.00762) (233, 0.00755) (234, 0.00749) (235, 0.00742) (236, 0.00736) (237, 0.0073) (238, 0.00724) (239, 0.00718) (240, 0.00712) (241, 0.00706) (242, 0.007) (243, 0.00694) (244, 0.00689) (245, 0.00683) (246, 0.00677) (247, 0.00672) (248, 0.00667) (249, 0.00661) (250, 0.00656) (251, 0.00651) (252, 0.00646) (253, 0.00641) (254, 0.00635) (255, 0.0063) (256, 0.00626) (257, 0.00621) (258, 0.00616) (259, 0.00611) (260, 0.00606) (261, 0.00602) (262, 0.00597) (263, 0.00593) (264, 0.00588) (265, 0.00584) (266, 0.00579) (267, 0.00575) (268, 0.00571) (269, 0.00567) (270, 0.00562) (271, 0.00558) (272, 0.00554) (273, 0.0055) (274, 0.00546) (275, 0.00542) (276, 0.00538) (277, 0.00534) (278, 0.0053) (279, 0.00527) (280, 0.00523) (281, 0.00519) (282, 0.00516) (283, 0.00512) (284, 0.00508) (285, 0.00505) (286, 0.00501) (287, 0.00498) (288, 0.00494) (289, 0.00491) (290, 0.00487) (291, 0.00484) (292, 0.00481) (293, 0.00478) (294, 0.00474) (295, 0.00471) (296, 0.00468) (297, 0.00465) (298, 0.00462) (299, 0.00459) (300, 0.00456)
        };
        \addlegendentry{Theorem~\ref{thm:main}}

    \end{semilogxaxis}
\end{tikzpicture}
\caption{Total variation after an $m$-shuffle on a $52$-card deck for $1 \le m \le 300$.}
\label{fig:tv-exact-approx}
\end{figure}


\subsection{Main result}
We denote by $\nu_n$ the uniform measure on the set of permutations on $n$ elements. We also denote by $\nu_{n,m}$ the measure on permutation induced by an $m$-shelf shuffle on $n$ elements, when starting from the identity. As already remarked, the composition rule from \cite{Diaconis_2013} allows to analyze the long-term behavior of the Markov chain obtained by repeatedly apply an $m$-shuffle via analyzing $\nu_{n,m}$ for large $m$. We now proceed to give a formal definition of $\nu_{n,m}$. In fact, it is easier to define an \emph{inverse} shelf shuffle, i.e., the time reversal of this chain (for the equivalence between different formulations, we refer the reader to \cite{Diaconis_2013}): cut a deck of $n$ cards into $2m$
piles according to a symmetric multinomial distribution. Reverse the order of
the even-numbered packets. Finally, riﬄe shuﬄe the $2m$ packets together by
the Gilbert–Shannon–Reeds (GSR) distribution, i.e., drop each card sequentially with probability proportional to packet size. The resulting permutation of the deck is such that its inverse is distributed precisely according to $\nu_{n,m}$.
\\
\\ 
We seek to analyze the behavior of the total variation distance
$$
\dtv(\nu_{n,m}, \nu_n)\coloneqq\frac{1}{2}\sum_{\pi}|\nu_n(\pi)-\nu_{n,m}(\pi)|,
$$
where the sum runs over all permutation of $n$ elements, in the case where $n$ and $m$ are large. Our main result answers a question from \cite{Diaconis_Fulman_2023} (Problem $2$ in section 11.5).
\begin{theorem}\label{thm:main}
Fix $c>0$. Then, 
$$
\dtv(\nu_{n,cn^{5/4}}, \nu_n)=
     1-2\Phi\left(\frac{1}{12c^2\sqrt{10}}\right) + O_c\left( \frac{1}{\sqrt{n}} \right).
$$
where $\Phi(x) = \int_{-\infty}^{x} e^{-t^2 / 2}\:dt / \sqrt{2 \pi}$ is the cumulative distribution function of a standard normal random variable.
\end{theorem}
In fact, 
\begin{remark}
In the proof, we will show $$
\dtv(\nu_{n,cn^{5/4}}, \nu_n)=\dtv\left(
            \mathcal{N}\left(
                -\frac{1}{6c^2 \sqrt{10}}, 1
            \right), \mathcal{N}(0, 1)
        \right)+O_c\left(\frac{1}{\sqrt n}\right)
     $$
 where $\mathcal{N}(\mu, \sigma^2)$ denotes the law of a normal random variable with mean $\mu$ and variance $\sigma^2$. The result then follows from the explicit form of the total variation distance between two normal with the same variance. The two normal distributions capture the behavior of a sufficient statistic for the Markov chain, namely the number of valleys (see Section \ref{sec:numvalleys}), with the main obstacle to mixing being given by a gap in the average number of valleys.
\end{remark}
\begin{remark}
By the composition rule given by \cite[][Corollary 4.2]{Diaconis_2013}, we can rephrase the result in the language of cutoff for Markov chains \cite{diaconis1996cutoff}: the $m$-shelf shuffle exhibits a cutoff at $\frac{5}{4} \log_{2m} n$ with window of constant size. It is interesting to compare this result with \cite{bayer92}. In light of the composition rule, repeating a $1$-shelf shuffle for $k$ times results in a $2^{k-1}$-shelf shuffle. For the choice of
$$
k=\frac{5}{4}\log_2n+\theta
$$
our main result shows that the total variation distance, in the limit as $n\rightarrow+\infty$, decreases from $1$ to $0$ as a function of $\theta\in\mathbb R$. For riffle shuffles a similar result holds with $5/4$ replaced by $3/2$, resulting in a larger mixing time.
\end{remark}


\subsection{Related literature}
We learned about the problem from Chapter $11$ in \cite{Diaconis_Fulman_2023}, where it was posed as an open question (Problem $2$ in section 11.5). In addition to providing an up-to-date survey of the mathematics of many shuffling schemes - including shelf shuffles -- they show versions of our result for the separation and $l_{\infty}$ distance. The first mention of shelf shuffle schemes comes from \cite{Diaconis_2013}, where the authors show the connection with the number of valleys of a permutation -- a sufficient statistic for the model. Some results on the number of valleys, including asymptotic behavior for uniformly random permutation and a useful connection wtih $P$-partitions, can be found in \cite{Fulman_2022, Fulman_2021}. Shelf shuffles can be thought as a variation of the more familiar riffle shuffle \cite{bayer92}. A recent comparison of the two schemes can be found in \cite{silverman2019progressive}. 
\\ \\
Beyond shuffling schemes, our result falls into the larger body of questions about the mixing time of Markov chains \cite{levin2017markov}. Let $\nu_{n,k}$ be the probability law induced by $k$ steps of an ergodic and aperiodic Markov chain with stationary distribution $\nu_n$. Given $\varepsilon\in (0,1)$, the $\varepsilon$-mixing time is defined by
$$
\tau_{n,\varepsilon}\coloneqq\min \{k\geq 0: \dtv(\nu_{n,k},\nu_{n})\leq \varepsilon\},
$$
i.e., the time that it takes for the chain to get within $\varepsilon$ of stationarity in total variation distance. A classical question is to understand the behavior for $\tau_{n,\varepsilon}$ for large $n$, and determine whether it depends on $\varepsilon$ in first approximation. When it does not, this is referred to as the cutoff phenomenon \cite{diaconis1996cutoff}. Rigorous proofs of cutoff are typically hard to obtain, since it requires a very detailed knowledge of the chain. Even rarer are instances in which the precise dependence on $\varepsilon$, thus going beyond the first order approximation of $\tau_{n,\varepsilon}$, can be understood. Notable exceptions in the context of card shuffling are the riffle shuffle \cite{bayer92} and random transpositions \cite{teyssier2020limit} (see also \cite{nestoridi2022limit} for a generalization). 

\subsection{Structure of the paper}
The rest of the paper is structured as follows. In Section \eqref{sec:numvalleys}, we analyze the number of valleys. The role of monotonicity is highlighted in \eqref{sec:monotone}, while sharp estimates for the mean and variances are in \eqref{sec:meanvar}. The main proof is then given in Section \eqref{sec:proofmainthm}. 
\section{Number of valleys}\label{sec:numvalleys}
Let $v(\pi)$ denote the number of valleys in a permutation
$\pi$, i.e., the number of positions $j$ such that $$\pi_{j-1}>\pi_j<\pi_{j+1}.$$ Recall that $\nu_n$ and $\nu_{n,m}$ respectively denote the uniform
measure and the $m$-shelf shuffle measure on the set of permutations.
Let $V_n = v(\pi)$ where $\pi$ is sampled from $\nu_n$, and
let $V_{n,m} = v(\pi)$ where $\pi$ is sampled from $\nu_{n,m}$.
In \cite{Diaconis_2013}, the authors show that $v$ is a sufficient statistic: after any number of steps, conditional on the resulting permutation having $k$ valleys, the resulting permutation is uniformly distributed over all permutation with $k$ valleys. In particular, since $v$ takes integer values inclusively between $0$ and
$u_n\coloneqq\lfloor (n - 1) / 2 \rfloor$, the
total variation between $\nu_n$ and $\nu_{n,m}$ is given by
\begin{equation}
\label{eq:tv-sufficient-statistic}
    \dtv(\nu_{n,m}, \nu_n)
    = \frac{1}{2} \sum_{k = 0}^{u_n}
        \left| \Pr(V_{n,m} = k) - \Pr(V_n = k) \right|.
\end{equation}
\subsection{Monotonicity results}\label{sec:monotone}
In what follows, we denote by $p_n(k), p_{n,m}(k)$ the p.m.f.\ of $V_n$ and $V_{n,m}$, respectively. From \cite[][Theorem 3.1]{Diaconis_2013}, they are related via
\begin{equation}
    p_{n,m}(k)= n! q_{n,m}(k) p_n(k)
\end{equation}
where
\begin{equation}
\label{eq:shelf-shuffle-scaling-term}
    q_{n,m}(k) 
    \coloneqq \sum_{r = k + 1}^{n - k}
        \frac{1}{m^n} \binom{m + n - r}{n} 
        \cdot \frac{1}{2^{n - 1 - 2k}}
            \binom{n - 1 - 2k}{r - 1 - k}
\end{equation}
is the probability that an $m$-shelf shuffle results in
a particular permutation $\pi$ if $\pi$ has $k$ valleys.
\cite[][Corollary 3.3]{Diaconis_2013}
showed that $q_{n,m}(v)$ is monotone decreasing in $v$. All together, this implies that the total variation distance, given by \eqref{eq:tv-sufficient-statistic}, can be rewritten as
\begin{equation}
\label{eq:shelf-shuffle-tv}
    \dtv(\nu_{n,m}, \nu_n) 
    = \max_k \left[\Pr(V_{n,m} \le k) - \Pr(V_n \le k)\right].
\end{equation}
The high-level overview of the proof for Theorem \ref{thm:main} is as follows.
For $V_n$, a central limit theorem is already known \cite{warren_seneta_1996}. The key idea is that $V_{n,m}$ is almost an exponential tilt of
$V_n$: in fact, we will show this by proving that the ratio $$\frac{q_{n,m}(k)}{q_{n,m}(k - 1)}$$ is nearly constant in $k$. To leverage this idea, we will need to exploit another monotonicity argument: we construct exponential tilts $V_{n,m}^-$ and $V_{n,m}^+$ of $V_n$ for which, with probability one, $$V_{n,m}^-\leq V_{n,m}\leq V_{n,m}^+.$$ Both $V_{n,m}^-$ and $V_{n,m}^+$ 
will be shown to satisfy central limit theorems with comparable means and variances, from which we will be able to give a sharp estimate to \eqref{eq:shelf-shuffle-tv}. Notice that, for a normal random variable, an exponential tilt simply corresponds to a shift in the mean.
\begin{remark}
This approach provides an alternative to the
arguments used in \cite[][Theorem 3.4]{Diaconis_2013} and 
\cite[][Section 4]{bayer92}. While they remark that their techniques should be applicable to our setting, the computations become
substantially more involved, while we find our approach to be more tractable.
\end{remark}
To make the heuristic precise, we will need to introduce some notation. For each $n$ and $m$, define 
\[
    \delta_{n,m}^- \coloneqq \min_{1 \le k \le u_n} 
        \frac{q_{n,m}(k)}{q_{n,m}(k - 1)};
    \qquad
    \delta_{n,m}^+ \coloneqq \max_{1 \le k \le u_n} 
        \frac{q_{n,m}(k)}{q_{n,m}(k - 1)}.
\]
Now, define the random variables $V_{n,m}^\pm$ as having
probability mass functions $f_{n,m}^\pm(k) p_n(k)$,
where $f_{n,m}^\pm$ are such that 
\[ 
    \frac{f_{n,m}^\pm (k)}{f_{n,m}^\pm (k - 1)} = \delta_{n,m}^\pm
\]
for each $1 \le k \le u_n$.

We claim that $V_{n,m}$ 
stochastically dominates $V_{n,m}^-$.
Write the p.m.f.\ of $V_{n,m}^-$ as $g_{n,m}(k) p_{n,m}(k)$.
Stochastic domination follows if $g_{n,m}$ is nonincreasing.
Indeed,
\begin{align*}
    \frac{g_{n,m}(k)}{g_{n,m}(k-1)}
    &= \frac{f_{n,m}^-(k) / n! q_{n,m}(k)}
        {f_{n,m}^-(k-1) / n! q_{n,m}(k-1)} \\
    &= \frac{f_{n,m}^-(k) / f_{n,m}^-(k-1)}
        {q_{n,m}(k) / q_{n,m}(k-1)} \\
    &= \frac{\delta_{n,m}^-}
        {q_{n,m}(k) / q_{n,m}(k-1)} \le 1
\end{align*}
for each $1 \le k \le u_n$ from the
definition of $\delta_{n,m}^-$. Likewise,
$V_{n,m}^+$ stochastically dominates $V_{n,m}$.

Of course, this is only useful provided that
$\delta_{n,m}^-$ and $\delta_{n,m}^+$ are not too far apart.
The following lemma confirms this property.
\begin{lemma}
\label{lem:shelf-shuffle-likelihood-ratio-ratio}
    Let $m = cn^{5/4}$ with $c > 0$ fixed. Then,
    for any integer $1 \le k \le u_n$, we have
    \begin{equation}
        \frac{q_{n,m}(k)}{q_{n,m}(k - 1)} = \exp\left( 
            -\frac{1}{4c^2 \sqrt{n}} + O_c \left( \frac{1}{n} \right) 
        \right)
    \end{equation}
    uniformly for the given values of $k$. In particular, 
    $$
    \delta_{n,m}^{\pm}=\exp\left( 
            -\frac{1}{4c^2 \sqrt{n}} + O_c \left( \frac{1}{n} \right)\right).
    $$
    
\end{lemma}
\begin{proof}
    For brevity, let $n_k \coloneqq n - 2k - 1$.
    By \eqref{eq:shelf-shuffle-scaling-term}, we have
    \begin{align*}
        m^n q_{n,m}(k - 1)
        &= \sum_{r = k}^{n - k + 1}
            \binom{m + n - r}{n} \cdot \frac{1}{2^{n_k}} \left[
                \frac{1}{4} \binom{n_k}{r - 2 - k}
                + \frac{1}{2} \binom{n_k}{r - 1 - k}
                + \frac{1}{4} \binom{n_k}{r - k}
            \right] \\
        &= \sum_{r = k + 1}^{n - k}
            \left[
                \frac{1}{4} \binom{m + n - r + 1}{n}
                + \frac{1}{2} \binom{m + n - r}{n}
                + \frac{1}{4} \binom{m + n - r - 1}{n}
            \right] \cdot \frac{1}{2^{n_k}} \binom{n_k}{r - 1 - k}
    \end{align*}
    where in the first step we applied Pascal's identity twice. In particular, this yields    
    \begin{align*}
        & m^n [q_{n,m}(k - 1) - q_{n,m}(k)] \\
        &\qquad = \sum_{r = k + 1}^{n - k}
            \left[
                \frac{1}{4} \binom{m + n - r + 1}{n}
                - \frac{1}{2} \binom{m + n - r}{n}
                + \frac{1}{4} \binom{m + n - r - 1}{n}
            \right] \cdot \frac{1}{2^{n_k}} 
                \binom{n_k}{r - 1 - k} \\
        &\qquad = \sum_{r = k + 1}^{n - k}
            \frac{1}{4} \binom{m + n - r - 1}{n - 2} 
            \cdot \frac{1}{2^{n_k}} \binom{n_k}{r - 1 - k} \\
        &\qquad = \sum_{r = k + 1}^{n - k}
            \frac{n(n - 1)}{4(m + n - r)(m + 1 - r)} 
                \binom{m + n - r}{n}
            \cdot \frac{1}{2^{n_k}} \binom{n_k}{r - 1 - k}.
    \end{align*}
    Let $j \coloneqq r - (n + 1) / 2$. Then, from the choice of $m = cn^{5/4}$ and the fact that $j=O(n)$,
    \begin{align*}
        \frac{n(n - 1)}{4(m + n - r)(m + 1 - r)}
        &= \frac{n(n-1)/4}{(m-j)^2 + (n-1)^2 / 4} \\
        &= \frac{n^2}{4(m-j)^2} \left(
            1 + O_c\left( \frac{1}{\sqrt n} \right)
        \right) \\
        &= \frac{n^2}{4m^2} \left(
            1 + \frac{j}{m} \left( 1 - \frac{j}{m} \right)^{-1}
        \right)^2 \left(
            1 + O_c\left( \frac{1}{\sqrt n} \right)
        \right) \\
        &= \frac{n^2}{4m^2} \left(
            1 + \frac{j}{m} + O \left( \frac{n^2}{m^2} \right)
        \right)^2 \left(
            1 + O_c\left( \frac{1}{\sqrt n} \right)
        \right) \\
        &= \frac{1}{4} \left( 
            \frac{n^2}{m^2} + \frac{2n^2 j}{m^3}
        \right) \left(
            1 + O_c\left( \frac{1}{\sqrt n} \right)
        \right).
    \end{align*}
    Substituting this into
    the previous expression and dividing through by $$m^n q_{n,m}=\sum_{r = k + 1}^{n - k}
         \binom{m + n - r}{n} 
        \cdot \frac{1}{2^{n - 1 - 2k}}
            \binom{n - 1 - 2k}{r - 1 - k},$$
    we can write 
    \begin{align*}
        & \frac{q_{n,m}(k - 1)}{q_{n,m}(k)} - 1 \\
        &\qquad = \frac{1}{4} \sum_{r = k + 1}^{n - k}
            \left( 
                \frac{n^2}{m^2} + \frac{2n^2 j}{m^3}
            \right) \left[
                \frac{1}{m^n q_{n,m}(k)}
                \cdot \binom{m + n - r}{n}
                \cdot \frac{1}{2^{n_k}} \binom{n_k}{j + n_k/2}
            \right] \left(
                1 + O_c\left( \frac{1}{\sqrt n} \right)
            \right) \\
        &\qquad = \frac{1}{4c^2 \sqrt{n}}
            + O_c\left( \frac{1}{n} \right) 
            + O_c(n^{-7/4}) \sum_{r = k + 1}^{n - k}
            j \left[
                \frac{1}{m^n q_{n,m}(k)}
                \cdot \binom{m + n - r}{n}
                \cdot \frac{1}{2^{n_k}} \binom{n_k}{j + n_k/2}
            \right].
    \end{align*}
    We now argue that the last term is also $O_c\left(1/n\right)$. To do so, observe that the term in square brackets, as a function of $j$,
    is the p.m.f.\ $p_X(j)$ of the random variable $X$ obtained by starting with
    $\mathrm{Bin}(n_k, 1/2)$, shifting it to have mean $0$, and
    tilting it by a factor proportional to $\binom{m + n - r}{n}$. Therefore the
    last sum is $\Ex[X]$. Let
    \begin{align*}
        \rho_{n,m}^- \coloneqq \min_{1 \le r \le n}
            \frac{\binom{m + n - r}{n}}{\binom{m + n - (r - 1)}{n}};
        \qquad
        \rho_{n,m}^+ \coloneqq \max_{1 \le r \le n}
            \frac{\binom{m + n - r}{n}}{\binom{m + n - (r - 1)}{n}}
    \end{align*}
    and define random variables $X^-$ and $X^+$ such that
    $X^\pm$ have p.m.f.\ $h^\pm(j) p_X(j)$ with $h^\pm(j)$ having
    common ratio $\rho_{n,m}^\pm$. By the same argument involving
    $V_{n,m}^\pm$, $X$ is stochastically bounded by $X^\pm$, which
    are exponential tilts of a shifted binomial random variable.
    Moreover,
    \[
        \frac{\binom{m + n - r}{n}}{\binom{m + n - (r - 1)}{n}}
        = \frac{m-r+1}{m+n-r+1}
        = 1 - \frac{n}{m+n-r+1}
        = \exp\left[ -\frac{n}{m} + O\left( \frac{n^2}{m^2} \right) \right]
    \]
    using $r \le n = o(m)$ and $e^x=1+x+O(x^2)$.
    Recalling that a $\theta$-tilted $\mathrm{Bin}(n, p)$ distribution
    is a $\mathrm{Bin}(n, pe^\theta / (1 - p + pe^\theta))$
    distribution, we have that $X^\pm + n_k / 2$ are
    $\mathrm{Bin}(n_k, 1/2 - n/4m + O(n^2/m^2))$ random variables.
    Hence
    \begin{align*}
        \Ex[X^\pm] = -\frac{n}{4m} n_k (1 + o_c(1)) = O_c(n^{3/4})
        \implies \Ex[X] = O_c(n^{3/4}),
    \end{align*}
    and thus
    \begin{align*}
        \frac{q_{n,m}(k - 1)}{q_{n,m}(k)} - 1
        &= \frac{1}{4c^2 \sqrt{n}} + O_c\left( \frac{1}{n} \right).
    \end{align*}
    Rearranging and using
    $e^x=1+x+O(x^2)$ gives the desired result.
\end{proof}

\subsection{On the mean and variance}\label{sec:meanvar}
We will need some control on the means and variances of $V_n$ and $V_{n,m}$, as well as showing a central limit theorem for both of them. For the former, \cite{warren_seneta_1996} shows that
\begin{equation}
\label{eq:uniform-clt}
    \left|
        F_{V_n}(k) - \Phi\left(
            \frac{k - \mu_n}{\sigma_n}
        \right)
    \right| = O\left( \frac{1}{\sqrt{n}} \right)
\end{equation}
uniformly in $k$,
where $$\mu_n \coloneqq \Ex[V_n] = \frac{n-2}{3}, \quad \sigma_n^2 \coloneqq \Var(V_n) = 
\frac{2n+2}{45}.$$
Their argument leverages the property that
$V_n$ has the same distribution as 
$X_0 + \cdots + X_{u_n}$, where
the $X_i$ are independent Bernoulli random variables with some parameters $p_i$, with $0<p_i<1$.\\ \\ %
We now turn our attention to $V_{n,m}$. While it cannot be represented directly as the sum of i.i.d.\ random variables, we will use our stochastic domination construction to get around the issue. In particular,
we will prove a central limit theorem for 
$V_{n,m}^{\pm}$, which we recall
bound $V_{n,m}$ from below and above. Since $V_{n,m}^\pm$ are respectively
$(\log \delta_{n,m}^\pm)$-tilted $V_n$'s, they
have the same distribution as the sum
of $(\log \delta_{n,m}^\pm)$-tilted $X_i$'s
and thus we can apply \cite[][Lemma 2]{warren_seneta_1996}.
For this to be useful, we need estimates of the
mean and variance of $V_{n,m}^\pm$,
which the following lemma provides.

\begin{lemma}
\label{lem:valley-stoc-bound-mean-var}
    Let $m = cn^{5/4}$ with $c > 0$ fixed. Then
    \[
        \Ex[V_{n,m}^\pm] = \mu_n - \sqrt{n} / 90c^2 + O_c(1);
        \qquad
        \Var(V_{n,m}^\pm) = \sigma_n^2 + O_c(\sqrt{n}).
    \]
\end{lemma}

\begin{proof}
Let $K_n(x)\coloneqq\log \mathbb E[e^{xV_n}]$ be the cumulant generating function of $V_n$. Since $V^{\pm}_{n,m}$ are exponential tilts of $V_n$ by a factor $\log\delta^{\pm}_{n,m}$, their c.g.f.s are given by
    $$K^{\pm}_n(x)=K_n(\log \delta_{n,m}^{\pm} + x) - K_n(\log \delta_{n,m}^{\pm}).$$ In particular, we can obtain cumulants of $V^{\pm}_{n,m}$ by differentiating $K_n$ and evaluating at $\log \delta_{n,m}^{\pm}$.
    The independent Bernoulli decomposition of $V_n$ indicates
    that
    \[
        K_n(x)= \sum_{i = 0}^{u_n} \log(1 - p + pe^x)
    \]
    where $0 < p_i < 1$ for each $i$. From \cite{warren_seneta_1996},
    we know that 
    $$
    \mathbb E[V_n]=K'_{n}(0)=\frac{n}{3}+O(1), \quad \Var(V_{n})=K_n^{''}(0)=\frac{2n}{45}+O(1).
    $$
    The idea now is to use a Taylor approximation to the third order with Lagrange remainder, and exploiting our detailed knowledge of $\delta_{n,m}^{\pm}$. To this aim, observe that by means of Lemma \ref{lem:shelf-shuffle-likelihood-ratio-ratio}, we can write
$$
\log \delta^{\pm}_{n,m}=-\frac{1}{4c^2\sqrt n}+O_c\left(\frac{1}{n}\right).
$$
As for the third derivative, we can use the bound $$u_n=O(n)$$ and the explicit computation to deduce
$$
\sup_{-\log \delta_{n,m}^-\leq x\leq 
\log \delta^+_{n,m}, p\in [0,1]}\frac{d^3}{dx^3}\log(1-p+pe^x)=O_c(1).
$$
All together, this allows to bound
\begin{align*}
M \coloneqq{}& \sup_{-\log \delta_{n,m}^-\leq x\leq 
\log \delta^+_{n,m}}\left|\frac{d^3K_n}{dx^3}(x)\right|\\
={}& O_c(n).
\end{align*}
Therefore, we obtain for the mean
\begin{align*}
\mathbb E[V_{n,m}^{\pm}]-\mathbb E[V_n]&= (K^{\pm}_n)'(0)-K'_{n}(0)\\&=
K'_n\left(\log\delta^{\pm}_{n,m}\right)-K'_n(0)
\\&=
-\frac{K^{''}_n(0)}{4c^2\sqrt n}+O_c\left(\frac{1}{n}M\right)\\&=-\frac{\sqrt n}{90c^2}+O_c\left(1\right)
\end{align*}
and similarly, for the variances,
\begin{align*}
\Var(V_{n,m}^{\pm})-\Var(V_n)&=(K^{\pm}_n)^{''}(0)-K^{''}_{n}(0)\\&=K^{''}_n(\log\delta^{\pm}_{n,m})-K^{''}_n(0) \\& =O_c\left(\frac{1}{\sqrt n}M\right)\\&=O_c(\sqrt n)
\end{align*}
thus concluding the proof.
\end{proof}
 
\section{Proof of main theorem}\label{sec:proofmainthm}
We now combine all our findings to prove our main result.



\begin{proof}[Proof of Theorem~\ref{thm:main}]
    Let $\mu_{n,m}^\pm \coloneqq \Ex[V_{n,m}^\pm]$
    and $(\sigma_{n,m}^\pm)^2 \coloneqq \Var(V_{n,m}^\pm)$.
    From \cite[][Lemma 2]{warren_seneta_1996} we can deduce a central limit theorem for $V^{\pm}_{n,m}$, and thanks to Lemma \eqref{lem:valley-stoc-bound-mean-var} we can estimate the error term as
    \begin{align*}
        \left|
            F_{V_{n,m}^\pm}(k) - \Phi\left(
                \frac{k - \mu_{n,m}^\pm}{\sigma_{n,m}^\pm}
            \right)
        \right|
        \le \frac{6}{\sigma_{n,m}^\pm}
        = O_c \left( \frac{1}{\sqrt{n}} \right)
    \end{align*}
    uniformly in $k$. Fix $x\in \mathbb R$, and let  
$$
k=\mu_n+x\sigma_n.
$$
From Lemma \ref{lem:valley-stoc-bound-mean-var}, we deduce that 
$$
\frac{k - \mu_{n,m}^\pm}{\sigma_{n,m}^\pm}=x\frac{\sigma_{n}}{\sigma^{\pm}_{n,m}}+\frac{\mu_n-\mu_{n,m^{\pm}}}{\sigma_{n,m}^{\pm}}=x\left(1+O_{c}\left(\frac{1}{\sqrt n}\right)\right)+\frac{1}{6\sqrt {10} c^2}+O_c\left(\frac{1}{\sqrt n}\right).
$$
In particular, from the central limit theorem for $V_n$ and $V^{\pm}_{n,m}$ and the stochastic domination of $V_{n,m}$,
$$
F_{V^{\pm}_{n,m}}(k)-F_{V_n}(k)=\Phi\left( x + \frac{1}{6c^2 \sqrt{10}} \right)
            - \Phi(x) + O_c\left( \frac{1}{\sqrt{n}} \right).
$$
with a uniform error as long as $x$ ranges in a compact set. Notice that the right side is maximized at
$$
x=\frac{1}{12\sqrt{10} c^2},
$$
resulting in a maximum of 
$$
\Phi\left(-\frac{1}{12c^2\sqrt{10}}\right)-\Phi\left(\frac{1}{12c^2\sqrt{10}}\right)=1-2\Phi\left(-\frac{1}{12c^2\sqrt{10}}\right).
$$
It remains to control the tails of $V_n$ and $V^{\pm}_{n,m}$. A Chebyschev inequality gives
$$
\mathbb P(|V_n-\mu_n|\geq \sigma_n |x|)\leq \frac{1}{x^2}<\frac{1-2\Phi\left(-\frac{1}{12c^2\sqrt{10}}\right)}{2}
$$
for $|x|$ sufficiently large. A similar argument works for $V^{\pm}_{n,m}$, exploiting the mean and variance asymptotic from Lemma \ref{lem:valley-stoc-bound-mean-var}. Since our monotonicity result guarantees that for all $k$
$$
F_{V^{+}_{n,m}}(k)\leq F_{V_{n,m}}(k)\leq F_{V^-_{n,m}}(k),
$$
we conclude that for $m=cn^{5/4}$ 
$$
\max_k \left[\Pr(V_{n,m} \le k) - \Pr(V_n \le k)\right]=1-2\Phi\left(-\frac{1}{12c^2\sqrt{10}}\right)+O_c\left(\frac{1}{\sqrt n}\right).
$$
Because of \eqref{eq:shelf-shuffle-tv}, we thus conclude
\begin{align*}
\dtv(\nu_{n,cn^{5/4}}, \nu_n)&=\max_{k}\left[\mathbb P(V_{n,m}\leq k)-\mathbb P(V_n\leq k)\right]\\&=1-2\Phi\left(\frac{1}{12c^2\sqrt{10}}\right)+O_c\left(\frac{1}{\sqrt n}\right).
\end{align*} 
\end{proof}

\printbibliography

\end{document}